\documentclass[12pt,reqno,a4paper]{article}
\usepackage[latin1]{inputenc}
\usepackage[T1]{fontenc}
\usepackage{amsthm}
\usepackage{amsfonts}
\usepackage{amsmath}
\usepackage{amssymb}
\numberwithin{equation}{section}
\newcommand{\pO}{\partial\Omega}
\newcommand{\eps}{\varepsilon}
\newtheorem{theor}{Theorem}[section]
\newtheorem{propo}{Proposition}[section]
\newtheorem{lem}{Lemma}[section]

\newtheorem{rem}{Remark}[section]
\usepackage{bbm}
\usepackage{graphics}
\usepackage{url}
\usepackage{a4wide}
\usepackage{enumerate}
\usepackage{hyperref}
\usepackage{pgf}
\usepackage{color}

\date{}
\author{A. Attouchi\thanks{
Universit\'e Paris 13, Sorbonne Paris Cit\'e, Laboratoire Analyse, G\'eom\'etrie et Applications (UMR CNRS 7539), 99, avenue Jean-Baptiste Cl\'ement
93430 - Villetaneuse, France.}
\and
G.~Barles\thanks{Laboratoire de
  Math\'ematiques et Physique Th\'eorique (UMR CNRS 7350), F\'ed\'eration Denis
  Poisson (FR CNRS 2964), Universit\'e Fran\c{c}ois Rabelais, Parc de Grandmont,
  37200 Tours, France.}}
\title{Global Continuation beyond Singularities on the Boundary for a Degenerate Diffusive Hamilton-Jacobi Equation}
\begin{document}
\maketitle

\begin{abstract}
In this article, we are interested in the Dirichlet problem for parabolic viscous Hamilton-Jacobi Equations. It is well-known that the gradient of the solution may blow up in finite time on the boundary of the domain, preventing a classical extension of the solution past this singularity. This behavior comes from the fact that one cannot prescribe the Dirichlet boundary condition for all time and, in order to define a solution globally in time, one has to use ``generalized boundary conditions'' in the sense of viscosity solution. In this work, we treat the case when the diffusion operator is the p-Laplacian where the gradient dependence in the diffusion creates specific difficulties. In this framework, we obtain the existence and uniqueness of a continuous, global in time, viscosity solution. For this purpose, we prove a Strong Comparison Result between semi-continuous viscosity sub and super-solutions. Moreover, the asymptotic behavior of $\dfrac{u(x,t)}{t}$ is analyzed through the study of the associated ergodic problem.
\end{abstract}

\vspace{1cm} \noindent{\bf Key-words : }Viscous Hamilton-Jacobi
Equations, generalized Dirichlet problem, maximum principle,
viscosity solutions, p-Laplacian.

\vspace{0.3cm}

\section{Introduction  and Main Results}
In this article we are interested in the  following generalized Dirichlet problem  for   second-order degenerate parabolic
partial differential equations
\begin{eqnarray}\label{pb}
 u_t-\text{div}\left([D u|^{p-2} D u\right) +| D u|^q&=&f(x,t)\quad\hbox{in   }\Omega\times (0, +\infty)\\
u(x,0)&=&u_0(x)\quad\hbox{on   }\overline{\Omega},\label{initial}\\  
u(x,t)&=&g(x,t)\quad\hbox{on   }\partial\Omega\times (0, +\infty), \label{boundary}
\end{eqnarray}
where $q>p\geq 2$,  $u_0$ and $g$ are continuous functions satisfying the compatibility condition
\begin{equation}\label{compa}
 u_0(x)=g(x,0)\quad\hbox{on   } \partial\Omega
\end{equation}
Most of works devoted to this degenerate  diffusive Hamilton-Jacobi equation concerned 
 the case where $\Omega=\mathbb{R}^N$,  providing results on well-posedness, gradient estimates and asymptotic behavior of either classical or weak solutions in the sense of distributions (see \cite{benlocal,  amourglobal,   superlinear, lauren1} and the references therein).

Some other works are concerned with the solvability of the Cauchy-Dirichlet problem.
They proved that, under suitable assumptions on $u_0$ and $g$, there exists a weak solution on some time interval $[0, T_{max}(u_0))$, with the property that its gradient blows up on the boundary $\pO$ while the solution itself remains bounded. We refer the reader to \cite{amala}  \cite{kawohl}  and \cite{laust} for the degenerate parabolic case and to \cite{zhan} for the uniformly parabolic case. This singularity is a difficulty to extend the solution past $T_{max}(u_0)$. A natural question is then: Can we extend the weak solution past   $t=T_{max}(u_0)$  and in which sense ?

Let us mention here that a result in this direction where the continuation beyond gradient blow-up does not satisfy  the original  boundary conditions was obtained in \cite{fila, winkler}.

Recently,  for the linear diffusion case ($p=2$), Barles and Da Lio \cite{daba} showed that such gradient blow-up is related to a loss of boundary condition and address the problem through a viscosity solutions approach. They proved a ''Strong Comparison Result'' (that is a comparison result  between \textsl{discontinuous} viscosity sub and supersolutions)
which allowed them to obtain the existence of a unique  continuous, {\em global in time viscosity solution} of \eqref{pb}--\eqref{boundary}, the Dirichlet boundary condition being understood in the {\em generalized sense of viscosity solution theory}. They also provided an explicit expression of the solution of \eqref{pb}--\eqref{boundary} in terms of a value function of some exit time control problem, which allows a simple explanation of the losses of boundary condition when it arises.

We recall that the formulation of the generalized Dirichlet boundary condition for \eqref{pb}--\eqref{boundary} in the viscosity sense  reads
\begin{equation}
\min \left( u_t-\text{div}\left(|D u|^{p-2} D u\right) +| D u|^q-f(x,t), u-g\right)\leq 0 \quad\hbox{on   } \pO\times(0,+\infty),
\end{equation}
and
\begin{equation}
\max \left( u_t-\text{div}\left(|D u|^{p-2} D u\right) +| D u|^q-f(x,t), u-g\right)\geq 0 \quad\hbox{on   }\pO\times(0,+\infty).
\end{equation}

Our first result mainly extends the investigation of \cite{daba} to the degenerate diffusion case $p>2$.
\begin{theor}\label{globalex}
 Assume that $q>p\geq 2$ and that $\Omega$ is a bounded domain with a $C^2$-boundary. For any $u_0 \in C(\overline{\Omega})$, $f \in C\left (\overline{\Omega}\times[0, T]\right)$ and $g\in C\left(\partial\Omega\times[0, T]\right)$ satisfying \eqref{compa}, there exists  a unique continuous solution $u$ of \eqref{pb}--\eqref{boundary} which is defined globally in time.
\end{theor}
As it is classical in viscosity solutions theory, the proof of Theorem~\ref{globalex} relies on a Strong Comparison Result (SCR in short), the existence of the global solution $u$ being an almost immediate consequence of the Perron's method introduced in the context of viscosity solutions by Ishii \cite{ishii} (see also \cite{dalio02}).

The most important difficulties in the proof of Strong Comparison Results come from  the formulation of the boundary condition in the viscosity sense, the discontinuity of the sub and the supersolution to be compared and the strong nonlinearity of the  Hamiltonian term $|D u|^q$.
A key argument in the proof of the SCR in \cite{daba} is the "cone condition" which is useful in the treatment of boundary points. Roughly speaking the ''cone condition'' holds if at any point $(\tilde{x}, \tilde{t})$ of the boundary $\pO\times (0, T)$, an usc  subsolution $u$ satisfies
 $u(\tilde{x}, \tilde{t})=\underset{k\to\infty}{\lim}\,  u(x_k, t_k)$ where $\left\{(x_k, t_k)\right\}_k$ is a sequence of  points of $\Omega\times(0,T)$ with the following properties
$$(x_k,t_k)\to(\tilde{x}, \tilde{t}) \quad\text{and} \quad d_{\pO}(x_k, t_k)\geq b\left(|x_k-\tilde{x}| + |t_k-\tilde{t}|\right),$$
where $b$ is a positive constant.

Our approach is slightly different: instead of directly proving the ''cone condition'' for any viscosity subsolution of \eqref{pb}--\eqref{boundary} as it was done in \cite{daba}, we use a combination of a $C^{0,\beta}$ regularity result for subsolutions of stationary problems, strongly inspired by the result of Capuzzo Dolcetta, Leoni and Porretta \cite{CLP}, together with a regularization by a sup-convolution in time. These arguments provide an approximation of the (a priori only usc) subsolution by a continuous subsolution, which automatically satisfies the ``cone condition'', allowing to borrow the methods of \cite{rouy} to conclude. 

The generalisation of the  $C^{0,\beta}$ regularity result of \cite{CLP} is the following.
\begin{theor}\label{holder}If $u$ is a locally bounded, usc viscosity subsolution
 of 
\begin{equation}\label{strong}
-\text{div} \left(|Du|^{p-2} Du\right) +|D u|^q\leq C \quad \hbox{in  }\Omega\; ,
\end{equation}
where $\Omega$ is an open subset of $\mathbb{R}^N$ and $C$ is a positive constant, and if $q>p\geq 2$, then $u\in C_{loc}^{0, \beta}(\Omega)$ with $\beta=\dfrac{q-p}{q-p+1}$.

Moreover, if $\Omega$ is a bounded domain with a $C^2$-boundary, then $u$ is bounded on $\overline{\Omega}$ and it can be extended as a $C^{0, \beta}$-function on $\overline{\Omega}$ and
\begin{equation}
 |u(x)-u(y)|\leq M |x-y|^{\beta}\quad \hbox{for all  }x, y \in \overline{\Omega},
\end{equation}
for some positive constant $M$ depending only on $p, q$, $C$ and $\partial\Omega$.
\end{theor}

The regularity result of \cite{CLP} was revisited in \cite{barle1}, where an interpretation was given in terms of state-constraint problems together with several possible applications. Our proof will rely on the arguments of \cite{barle1}.

A second motivation where such regularity results are useful, is the asymptotic behavior as $t\to+\infty$ of solutions of the evolution equation. For this purpose, one has first to study the ergodic (or additive eigenvalue) problem
\begin{equation}\label{erg}
 -\text{div} \left(|D u_{\infty}|^{p-2} D u_{\infty} \right)+|D u_{\infty}|^q-\tilde{f}(x)=c\quad\text{in  }\Omega,
\end{equation}
associated to a state-constraint boundary condition on $\partial\Omega$
\begin{equation}\label{constr}
  -\text{div}  \left(|D u_{\infty}|^{p-2} D u_{\infty} \right)+|D u_{\infty}|^q-\tilde{f}(x)\geq c \quad\text{on  } \partial\Omega.
\end{equation}
We recall that, in this type of problems, both the solution $u_{\infty}$ and the constant $c$
(the ergodic constant) are unknown. First we have the following result.
\begin{theor}\label{ergodic}
Assume that $\Omega$ is  a bounded domain with a $C^2$-boundary, $\tilde{f}\in C(\overline{\Omega})$ and $q >p \geq 2$,    then there exists a unique constant $c$  such that the state-constraints problem \eqref{erg}--\eqref{constr} has a continuous  viscosity solution $u_{\infty}$.
\end{theor}
A typical result  that connects the study of the ergodic problem to the large time behavior of the  solution  $u$ of  \eqref{pb}--\eqref{boundary}  is the following.
\begin{theor}Assume that $\Omega$ is  a bounded domain with a $C^2$-boundary, $u_0\in  C(\overline{\Omega})$, $g\in C(\partial\Omega)$ satisfying  \eqref{compa} and assume that $f(x,t)=\tilde{f}(x)$ with $\tilde{f}\in C(\overline{\Omega})$ and $q>p\geq 2$.  If $(c,u_{\infty})$ is the solution of   \eqref{erg}--\eqref{constr} and if $u$ is the unique viscosity solution of \eqref{pb}--\eqref{boundary}, then $u+c^{+}t$ is bounded, where $c^+ = max(c,0)$. In particular
$$\lim_{t\to\infty} \dfrac{u(x,t)}{t}=-c^+$$
uniformly on $\overline{\Omega}$.
\end{theor}
The next step in the study of the asymptotic behavior would be to show that $u(x, t)+ct \to u_{\infty}(x)$ as $t\to\infty$ where $u_{\infty}$ solves \eqref{erg}--\eqref{constr}. The main difficulty  to prove such more precise asymptotic behavior comes from the fact that  \eqref{erg}--\eqref{constr} does not admit a unique solution (\eqref{erg}--\eqref{constr} is invariant by addition of constants). Such results were obtained recently in \cite{tchamba} for the uniformly elliptic case $p=2$  through the use of the Strong Comparison Principle (i.e. a result which allows to apply the Strong Maximum Principle to the difference of solutions) and the Lipschitz regularity of $u_{\infty}$. But, for $p > 2$, such Strong Comparison Principle is not available since the equation is quasilinear and not semilinear. We recall that a Strong Maximum Principle is available for $p>2$, see \cite{bardi}. Another difficulty comes from  the proof of a strong comparison result for the steady problem in  case of an operator that does not fulfill  a monotonicity property,  even if there exits a strict subsolution.
Let us mention the works of \cite{laust, balaust} for more results on the asymptotic behavior of global solutions.

Finally we point out that it was shown in \cite{bptt} that the expected asymptotic behavior, namely $u(x, t)+ct \to u_{\infty}(x)$, is not always true in the $p=2$-case when the nonlinearity is sub quadratic in $Du$.

This article is organized as follows: in Section 2, we present the needed results on viscosity solutions for the stationary and evolution problems we consider; in particular, we analyze the losses of boundary conditions for subsolutions. In Section 3 we prove the H\"older regularity result of  Theorem \ref{holder}. In Section 4 we study the ergodic problem.  Section 5 is devoted to  the proof of Theorem \ref{globalex} and  the asymptotic behavior of solutions of the evolution equation.

\section{Preliminaries and Analysis of Boundary Conditions}
In this section we collect some preliminary properties of viscosity subsolutions (the boundary conditions being always understood in the viscosity sense) and we also formulate SCR under different forms, some of them being only useful as a step in the proof of the complete regularity result. These results are concerned with either problem \eqref{pb}--\eqref{boundary} or the following two nonlinear elliptic problem
\begin{equation}\label{operat}
 -(p-1)|Du|^{p-2}\sum_{\lambda_i(D^2 u)>0} \lambda_i(D^2 u)+ |D u|^{q}= C \quad\text{in  }\Omega,\qquad u=\tilde{g}\quad\text{in  }\partial\Omega.
\end{equation}
and
\begin{equation}\label{quasistation}
 -div\left(|D u|^{p-2} D u\right)+|D u|^q+\lambda u-\tilde{f}=0\quad\text{in  } \Omega,\qquad u=\tilde{g}\quad\text{in  }\partial\Omega.
\end{equation}
where  $q>p\geq 2$, $ C, \lambda\geq 0$, $\tilde{f}\in C(\overline{\Omega})$ and $\tilde{g}\in C(\partial \Omega)$.\\

From now on,  we assume that $\Omega$ is a smooth domain with a $C^{2}$-boundary. We define the distance from $x\in\overline{\Omega}$ to $\partial\Omega$ by $d_{\partial\Omega}(x) := \text{dist}\, (x, \partial\Omega)$. For $\delta>0$, we denote by
\begin{eqnarray}
&\Omega^{\delta}:=\left\{x \in\Omega\, |\, d_{\partial\Omega}(x)<\delta\right\},\\
&\Omega_{\delta}:=\left\{x \in\Omega\, |\, d_{\partial\Omega}(x)>\delta\right\}.
\end{eqnarray}
 As a consequence of the regularity of $\partial\Omega$, $d_{\partial\Omega}$ is a $C^2$-function in a neighborhood  $\Omega^{\delta}$ of the boundary for all $0<\delta\leq \delta_0$.
We denote by $d$ a $C^2$-function agreeing with $d_{\partial\Omega}$ in $\Omega^{\delta}$ such that $|D d(x)|\leq 1$ in $\Omega_{\delta}$. We also denote by $n(x)$ the $C^1$-function defined by $n(x)=-D d(x)$ in $\Omega^{\delta}$; if $x \in\partial\Omega$, then $n(x) $ is just the
unit outward normal vector to $\partial\Omega$
 at $x$.
 
 Our first result says that there is no loss of boundary conditions for the subsolutions, namely that the subsolutions satisfy the boundary condition in the classical sense.
\begin{propo}\label{boundcond}
Assume that $q>0$ and $p\geq 2$. We have the following
\begin{enumerate}[i)]
\item
 If  $u$ is a bounded, usc subsolution
 of \eqref{pb}--\eqref{boundary} on a time interval $(0,T)$, then
\begin{equation}
u\leq  g \quad\text{on}\quad \partial\Omega\times(0,T). 
\end{equation}
\item  If $u$ is a bounded, usc subsolution
 of \eqref{operat} or \eqref{quasistation} , then
\begin{equation}
u\leq  \tilde{g} \quad\text{on}\quad \partial\Omega.
\end{equation}
\end{enumerate}
\end{propo}
\begin{proof} We only give the proof  for the time dependent problem, the proof for the stationnary problems being similar.  We use a result of Da Lio \cite[Corollary 6.2]{dalio1}.  We denote by $\mathcal{S}^N$ the space of real symmetric $N\times N$ matrices. For $x\in\Omega$, $t\in (0,T)$, $\xi\in\mathbb{R}^N$ and $M\in\mathcal{S}^N$, we define the function $F$ by
$$F(x, t, \xi, M)=-|\xi|^{p-2}Tr(M)-(p-2)|\xi|^{p-4}\left\langle M\xi, \xi\right\rangle+|\xi|^q- f(x,t),$$
so that the equation can be written as $u_t+F(x, t, Du, D^2u)=0$.
From \cite{dalio1}, we know that, if $u(x_0, t_0)> g (x_0, t_0)$ at some point $(x_0, t_0)\in\partial\Omega\times (0, T)$, then the following conditions hold
\begin{eqnarray}
&&\underset{\substack{(y,t)\to(x_0,t_0)\\
\alpha\downarrow 0}}{\lim \inf}\left\{\left[\dfrac{o(1)}{\alpha}+ F\left(y,t,\dfrac{D d(y)+o(1)}{\alpha}, -\dfrac{ D d(y)\otimes D d(y)+ o(1)}{\alpha^2}\right)\right]\right\}\leq 0\nonumber\\
&&\underset{\substack{
(y,t)\to(x_0,t_0)\\
\alpha\downarrow 0
}}{\lim \inf} \left\{\left[\dfrac{o(1)}{\alpha}+ F\left(y,t,\dfrac{D d(y)+o(1)}{\alpha}, \dfrac{ D^2 d(y)+ o(1)}{\alpha}\right)\right]\right\}\leq 0.
\end{eqnarray}
But the first condition cannot hold since
$$F\left(y,t,\dfrac{D d(y)+o(1)}{\alpha}, -\dfrac{ D d(y)\otimes D d(y)+ o(1)}{\alpha^2}\right)\geq \dfrac{(p-1)}{\alpha^p}\left(1+o(1)\right)+ \dfrac{1-o(1)}{\alpha^q}-f(y,t),$$
 and the right hand side is going to $+ \infty$ as $\alpha\to 0$ since $p\geq 2$, $q>0$ and  all terms converge to $+\infty$.
\end{proof}
Let us point out that the above computation shows that there is no competition between the nonlinear Hamiltonian term and the  slow diffusion operator since they both produce positive contribution which prevent any loss of boundary conditions for the subsolution.\\

Next, we remark that there  cannot be loss of initial condition.
\begin{lem}
Assume that $q>p\geq2$, $f\in C\left( \overline{\Omega}\times [0,T]\right)$ and $u_0\in C(\overline{\Omega})$, $g\in C\left(\partial\Omega\times [0,T]\right)$ satisfy \eqref{compa}. Let $u$ and  $v$ be respectively a bounded usc  viscosity subsolution
 and a bounded lsc super-solution of\eqref{pb}--\eqref{boundary} then
\begin{equation}
\label{initcond}
u(x,0)\leq u_0(x)\leq v(x,0)\quad
\text{on}\quad\overline{\Omega}.
\end{equation}
\end{lem}
\textbf{Proof} Fix $x_0\in\overline{\Omega}$ and define for $\varepsilon>0$ and $C_{\varepsilon}>0$ the function $\phi_{\varepsilon}(x,t)$ by
$$ \phi_{\varepsilon}(x,t) =u(x,t)-\dfrac{|x-x_0|}{\varepsilon^2}-C_{\varepsilon} t. $$
This function attains  a global maximum on $\overline{\Omega}\times [0, T)$ at $x_{\eps}, t_{\eps}$.
Using the boundedness of $u$, it is easy to see that, for any $C_{\eps}>0$, $(x_{\eps}, t_{\eps})\to (x_0, 0)$ as $\eps\to0$. Arguing as in \cite{rouy}, choosing $C_{\eps}$ sufficiently large depending on $\eps$, we are left with $(x_{\eps}, t_{\eps})\in \left(\partial\Omega\times (0, T)\right)\cup \left(\overline{\Omega}\times\left\{0\right\}\right)$ and the two following possibilities
\begin{eqnarray*}
\text{either}\quad t_{\eps}=0 \quad&\text{and}&\quad u(x_{\eps}, 0)\leq u_0(x_{\eps}),\\
\text{or}\quad t_{\eps}>0, x_{\eps}\in\pO \quad&\text{and}&\quad u(x_{\eps}, t_{\eps})\leq g(x_      {\eps}, t_{\eps}).
\end{eqnarray*}

In either case, since $u(x_0, 0)\leq \phi_{\varepsilon}(x_{\eps},t_{\eps})\leq u(x_{\eps}, t_{\eps})$, we  get the desired  result for $u$ letting $\eps\to 0$ and using the continuity of $u_0$ and $g$. The argument for $v$ is similar.
\bigskip

Now we claim that under some assumptions  (set out below), a SCR holds for semicontinuous viscosity sub-and supersolutions of  \eqref{pb}--\eqref{boundary} or \eqref{operat} or \eqref{quasistation}. The proof being somehow technical we refer the reader to the appendice for a detailed proof of the  following two propostions.

\begin{propo}[\textbf{Parabolic SCR}]\label{parascr}
 Assume that $q>p\geq2$, $f\in C\left( \overline{\Omega}\times [0,T]\right)$ and $u_0\in C(\overline{\Omega})$, $g\in C\left(\partial\Omega\times [0,T]\right)$ satisfy \eqref{compa}. Let $u$ and  $v$ be respectively a bounded usc  viscosity subsolution
 and a bounded lsc super-solution of \eqref{pb}--\eqref{boundary}, then $u\leq v$ in $\Omega\times [0, T]$. Moreover, if we define $\tilde{u}$ on $\overline{\Omega}\times [0, T]$ by setting
\begin{equation}
\tilde{u}(x,t):=\left\{\begin{array}{ll}
\underset{\substack{(y,s)\to (x,t)\\
(y,s)\in\Omega\times(0,T)}}
{\lim\sup u(y,s)} &\quad\text{for all} \, (x,t)\in\partial\Omega\times(0,T]\\
u(x,t) &\qquad \text{otherwise},
\end{array}
\right.
\end{equation}
then $\tilde{u}$ remains an usc subsolution
 of \eqref{pb}--\eqref{boundary} and
\begin{equation}\label{compar}
\tilde{u}\leq v \quad\text{on  }\overline{\Omega}\times[0, T].
\end{equation}
\end{propo}

The stationary version of the SCR is used either in the proof of the $C^{0,\beta}$-regularity or for solving the ergodic problem.
\begin{propo}[\textbf{Elliptic SCR}]\label{ellipscr}
 Assume that $q>p\geq2$, $\tilde{f}\in C\left( \overline{\Omega}\right)$ and  $\tilde{g}\in C\left(\partial\Omega\right)$. 

\begin{enumerate}[(i)]
\item 
Let $u$ and $v$ be respectively a bounded usc  viscosity subsolution
 and a bounded lsc super-solution of \eqref{operat}. If $v$ is continuous on $\overline{\Omega}$ and is  a strict supersolution of \eqref{operat}, then
\begin{equation}\label{regcom}
u \leq v \quad
\text{on}\quad\overline{\Omega}.
\end{equation} 
\item
Let $u$ and  $v$ be respectively a bounded usc  viscosity subsolution
 and a bounded lsc super-solution of \eqref{quasistation}. Assume that either $\lambda>0$ or $\lambda=0$ and $v$ is  a strict supersolution. We define $\tilde{u}$ on $\overline{\Omega}$ by setting
\begin{equation}
\tilde{u}(x):=\left\{\begin{array}{ll}
\underset{\substack{y\to x\\
y\in\Omega}}
{\lim\sup u(y)} &\quad\text{for all  } x\in\partial\Omega\\
u(x) &\qquad \text{otherwise},
\end{array}
\right.
\end{equation}
then $\tilde{u}$ remains an usc subsolution
 of \eqref{pb}--\eqref{boundary} and
\begin{equation}\label{ellipcompar}
\tilde{u}\leq v \quad
\text{on}\quad\overline{\Omega}.
\end{equation} 
\end{enumerate}
\end{propo}

\section{H\"older Regularity of Viscosity Subsolutions for the Degenerate Elliptic Problem}

In this section we are going to prove that  equation of type \eqref{strong}  enters into the general framework described in \cite{barle1} which allows us to state that, if  $u$ is a locally bounded, usc viscosity  subsolution
 of \eqref{strong}, then $u$ is H\"{o}lder continuous with exponent $\beta=\dfrac{q-p}{q-p+1}$. The key point is that the strong growth of the first order term balances the degeneracy of the  second order term, providing a control on $|Du|$.

\medskip

\noindent\textbf{Proof of Theorem \ref{holder}.} If $u$ is a subsolution
 of \eqref{strong}, then it is a subsolution
 in $B_r(x)=\left\{ y\in \mathbb{R}^N; \, |y-x|<r\right\}$ of the simpler equation
$$-(p-1)|Du|^{p-2}\sum_{\lambda_i(D^2 u)>0} \lambda_i(D^2 u)+ |D u|^{q}\leq C.$$

Now we are going to check the required hypotheses in \cite{barle1}.
\medskip

\textbf{H1.} For $0<r<1$,  $s\in\mathbb{R}^N$ and $M\in \mathcal{S}^N$, $\mathcal{S}^N$ denoting the space of $N\times N$ real valued symmetric matrices, define
the function $G_r(s,M)$ by 
$$G_r(s,M):=-(p-1) |s|^{p-2}\sum_{\lambda_i(M)>0} \lambda_i(M)+ |s|^{q}- C.$$
Then, for any $x\in\Omega$ with $d_{\partial\Omega} (x)\geq  r$, $G_r(Du, D^2 u)\leq 0$ in $B_r(x)$.
\medskip

\textbf{H2.} There exists a super-solution up to the boundary  $w_r\in C\left(\overline{B_r(0)}\right)$ such that $w_r(0)=0$, $w_r(x)\geq 0$ in $B_r(x)$ and 
\begin{equation}
G_r\left( D w_r, D^2 w_r\right)\geq \eta_r>0 \quad\text{on}\quad \overline{B_r(0)}\backslash\left\{0\right\},
\end{equation}
for some $\eta_r>0$.\\

Despite the  construction of the functions $w_r$ is a rather easy adaptation of \cite{barle1}, we reproduce it for the sake of completeness and for the reader's convenience.
In order to build  $w_r$,  we first build $w_1$ and then use the scale invariance of the equation. To do so, we borrow arguments from \cite{barle1}. For $C_1,C_2>0$ to be chosen later on and for $\beta=\dfrac{q-p}{q-p+1}$, we consider the function
$$w_1(x):=\dfrac{C_1}{\beta}|x|^{\beta}+ \dfrac{C_2}{\beta}\big( d^{\beta}(0)-d^{\beta}(x)\big),$$
where $d(x)=1-|x|$ on $B_1(0)\backslash B_{1/2}(0)$ and we regularize it in $B_{1/2}(0)$ by changing it into $h(1-|x|)$ where $h$ is a smooth, non-decreasing and concave  function such that $h(s)$ is constant for $s\geq 3/4$ and $h(s)=s$ for $s\leq 1/2$.
 Obviously we have $w_1(0)=0$, $w_1\geq 0$ in $\overline{B_1(0)}$ and $w_1$ is smooth in $\overline{B_1(0)}\backslash\left\{0\right\}$.
 
We first remark that $-(p-1)|D w_1 (x)|^{p-2}\lambda_i(D^2 w_1(x)) +|D w_1(x)|^q$ can be written as
$$
|D w_1 (x)|^{p-2}\left[-(p-1)\lambda_i(D^2 w_1(x)) +|D w_1(x)|^{q-p+2}\right]\; .
$$
Therefore, in order to prove the claim, we are going to show that, for $C_1,C_2>0$ large enough, the bracket is positive and bounded away from $0$ and that $|D w_1 (x)|^{p-2}$ remains large.

Computing the derivatives of $w_1$ in $B_1(0)\backslash\left\{0\right\}$, we have
\begin{eqnarray*}
D w_1(x)&=& C_1|x|^{\beta-2}x- C_2d^{\beta-1}(x)D d(x),\\
D^2 w_1(x)&=& C_1|x|^{\beta-2} Id +(\beta-2)C_1|x|^{\beta-4} x\otimes x\\
&-& C_2 d^{\beta-1}(x) D^2 d(x)-(\beta-1) C_2 d^{\beta-2}(x) D d(x)\otimes D d(x).
\end{eqnarray*}
Using that $-D d(x)=\mu(x) x$ for some  $\mu(x)\geq 0$ and that  $q>p>2$, we have
\begin{eqnarray*}
|D w_1(x)|^{q-p+2}&=&\left(|C_1|x|^{\beta-2}x|+| C_2 d^{\beta-1}(x) D d(x)|\right)^{q-p+2}\\
             &\geq& | C_1|x|^{\beta-2}x|^{q-p+2}+| C_2 d^{\beta-1}(x) D d(x)|^{q-p+2}\\
             &=& C_1^{q-p+2}|x|^{(\beta-1)(q-p+2)} + C_2^{q-p+2}d^{(\beta-1)(q-p+2)}(x) |D d(x)|^{q-p+2},
\end{eqnarray*}
and 
\begin{equation*}
|D w_1(x)|^{p-2}\geq C_1^{p-2}|x|^{(\beta-1)(p-2)} + C_2^{p-2}d^{(\beta-1)(p-2)}(x) |D d(x)|^{p-2}.
\end{equation*}
Using that $d(x)=h(1-|x|)$, with $h$ being $C^2$, non-decreasing and concave, $0<\beta<1$, we have
$$D^2 w_1(x) \leq  C_1|x|^{\beta-2} Id + C_2 d^{\beta-1}(x)\left(\frac{h'}{|x|} Id -h''\frac{x}{|x|}\otimes\frac{x}{|x|}\right)+(1-\beta) C_2 d^{\beta-2}(x) D d(x)\otimes D d(x),$$
and $$\lambda_i(D^2 w_1(x))\leq C_1|x|^{\beta-2}  + C_2 d^{\beta-1}(x)\left(\frac{h'}{|x|} -h''\right)+(1-\beta) C_2 d^{\beta-2}(x) |D d(x)|^2.$$
At this point, it is worth noticing that because of the properties of $h$, the term $\left(\frac{h'}{|x|} -h''\right)$ is bounded.

These properties imply that, we can (almost) consider the two terms (in $|x|$ and in $d(x)$) separately. Since $(\beta-1)(q-p+2)= (\beta-2)$, the $\dfrac{C_1}{\beta} |x|^{\beta}$ term yields
$$-(p-1) C_1 |x|^{\beta-2}+ | C_1|x|^{\beta-2}x|^{q-p+2}=|x|^{\beta-2}\left(-(p-1) C_1 + C_1^{q-p+2}\right).$$
By choosing $C_1$  large enough, we can have for any $K_1>0$
$$|x|^{\beta-2}\left(-(p-1) C_1 + C_1^{q-p+2}\right)\geq K_1 |x|^{\beta-2}\quad\text{in}\, B_1(0)\backslash\left\{0\right\}.$$
On the other hand the $\dfrac{C_2}{\beta} \big( d^{\beta}(0)-d^{\beta}(x)\big)$ term yields
\begin{equation}\label{secterm}
-(p-1) C_2 d^{\beta-1}(x)\left(\frac{h'}{|x|} -h''\right)+(\beta-1)(p-1) C_2 d^{\beta-2}(x) |D d(x)|^2
+ C_2^{q-p+2} |d^{\beta-1} D d(x)|^{q-p+2}.
\end{equation}
We have to consider two cases: either $|x|\geq \dfrac{1}{2}$ and then $h'=1$, $h''=0$ and $D d(x)=-\dfrac{x}{|x|}$; hence the above quantity is given by 
$$-(p-1)C_2 d^{\beta-1}(x)\left(\frac{1}{|x|}\right)-(p-1)(1-\beta) C_2 d^{\beta-2}(x)+ C_2^{q-p+2} d^{(\beta-1)(q-p+2)}.$$
Recalling that $(\beta-1)(q-p+2)=(\beta-2)$, then  for $C_2$ large enough  we have  for any $K_2>0$
$$ -(p-1)C_2 d^{\beta-1}(x)\left(\frac{1}{|x|}\right)-(p-1)(1-\beta) C_2 d^{\beta-2}(x)+ C_2^{q-p+2} d^{(\beta-1)(q-p+2)}\geq K_2 d^{\beta-2}(x).$$
Now for $|x|\leq \dfrac{1}{2}$, the quantity \eqref{secterm} coming from the $d(x)$-term is bounded and can be controlled by the $|x|$-term. Hence, for any constant $C>0$, choosing first $C_2$ large enough and then $C_1$ large enough,
 we have in $\overline{B_1(0)}\backslash\left\{0\right\}$ 
\begin{eqnarray*}
-(p-1)|D w_1 (x)|^{p-2}\lambda_i(D^2 w_1(x)) +|D w_1(x)|^q&\geq& |D w_1(x)|^{p-2}\left( K_1 |x|^{\beta-2}
+K_2 d^{\beta-2}(x)\right)\\
&\geq &\left( K_1 |x|^{(\beta-1)(p-1)-1}\right) \geq C.
\end{eqnarray*}
Next we  set
$$w_r(x):=r^{\beta} w_1\left(\dfrac{x}{r}\right).$$
It is easy to check that  for $0<r\leq 1$, $G( D w_r,D^2 w_r)\geq r^{(\beta-1)(p-1)-1}C-C\geq 0$ on $\overline{B_r(0)}\backslash\left\{0\right\}$.
\medskip

\textbf{H3.}  Comparison result.  Let $v$ be any bounded usc  viscosity subsolution
 of $G_r( Dv, D^2v)\leq0$ in $B_r(0)\backslash\left\{0\right\}$ then
\begin{equation}
v(y)\leq v(x)+ r^{\beta}w_1\left(\dfrac{y-x}{r}\right).
\end{equation}
 We use the fact that $v(0)+w_r(x)$ is a  \textbf{strict} super-solution up to the boundary and that it is a \textbf{continuous} function. It follows that the comparison  is a direct consequence of Proposition \ref{ellipscr}.

Since the hypotheses are satisfied, we can apply Proposition 2.1 of  \cite{barle1}  to obtain the $C^{0,\beta}$ regularity of subsolutions, both locally and globally with further assumptions on $\Omega$.

\begin{rem}As far as the exponent $\beta$ is concerned, the value  is the best one can expect in the assumption of the above theorem (see \cite{CLP}).\\
It is well-known that the degeneracy of the $p$-Laplacian is an an obstruction to the solvability of the Dirichlet problem in the classical sense. The presence of the strongly non-linear term with $q>p$ is another source of obstruction, even in the uniformly elliptic case since examples of boundary layers can occur \cite{daba, lasrylions}. By the previous result, we know that every continuous solution to \eqref{strong} is H\"older  continuous up to the boundary. Hence, a necessary condition in order that the solution can attain continuously the boundary data $g$ is the existence of some $C\geq 0$ such that
$$|g(x)-g(y)|\leq C[x-y|^{\beta}\quad \hbox{for all   } x, y \in\partial\Omega, \quad\beta=\dfrac{q-p}{q-p+1}.$$
 For the uniformly elliptic case $p=2$,  a more detailed study including several gradient bounds and applications can be  found in \cite{lasrylions}.

\end{rem}

As an application of the previous regularity result,  we consider the generalized Dirichlet problem consisting in solving \eqref{quasistation}.

\begin{theor}\label{fab}
Let  $\Omega\subset \mathbb{R}^N$ be a bounded domain  with a $C^2$-boundary.  Assume that $q>p\geq 2$,   $\tilde{f}\in C(\overline{\Omega})$, $\tilde{g}\in C(\pO)$ and $\lambda>0$.
Let $u$ and $v$ be respectively a bounded usc  subsolution
  and a bounded lsc super-solution of \eqref{quasistation} with $u$ satisfying for $x\in\pO$
$$u(x)=\underset{\substack{y\to x\\ y\in\Omega}}{\limsup}\quad u(y).$$
Then, $u\leq v$ on $\overline{\Omega}$. Moreover Problem~\eqref{quasistation} has a unique viscosity solution which belongs to $C^{0,\beta}(\overline{\Omega})$.
\end{theor}

\begin{proof} For the comparison part, Theorem \ref{holder} implies that $u$ is H\"older continuous, hence the  comparison $u\leq v$ is a  direct consequence of Proposition \ref{ellipscr}. Once noticed that $-\left(\lambda^{-1}||\tilde{f}||_{L^{\infty}}+||\tilde{g}||_{L^{\infty}}\right)$  and $+\left(\lambda^{-1}||\tilde{f}||_{L^{\infty}}+||\tilde{g}||_{L^{\infty}}\right)$ are respectively sub and super-solution, we can apply the Perron's method with the version up to the boundary  (see \cite{dalio02}). Since a solution is also a subsolution, the 
H\"older regularity is a direct consequence of Theorem \ref{holder}.

\section{The Ergodic Problem}
\subsection{Existence of the pair $(c, u_{\infty})$}

In this part we study the existence of a pair $(c, u_{\infty})\in \mathbb{R}\times C(\overline{\Omega})$ for which $u_{\infty}$ is a viscosity solution of the state-constraints problem \eqref{erg}-\eqref{constr}, to gather with the uniqueness of the ergodic constant $c$. For this purpose, we introduce a $\lambda u$-term in the equation, as it is classical, with the aim to let $\lambda$ tend toward $0$. This key step is described by the following Lemma.
\begin{lem}
 Let $\tilde{f}\in  C(\overline{\Omega})$ and $\beta=\dfrac{q-p}{q-p+1}$.  For $0<\lambda<1$ and $q>p$, there exists  a unique viscosity solution  $u_{\lambda}\in C^{0, \beta}(\overline{\Omega})$ of the state constraint problem 
\begin{equation}\label{clem1}
 -\text{div}\, (|D u_{\lambda}|^{p-2} D u_{\lambda} )+|D u_{\lambda}|^q+\lambda u_{\lambda}=\tilde{f}(x)\quad\hbox{in   }\Omega,
\end{equation}
\begin{equation}\label{clem2}
-\text{div}\, (|D u_{\lambda}|^{p-2} D u_{\lambda} )+|D u_{\lambda}|^q+\lambda u_{\lambda}\geq \tilde{f}(x) \quad\hbox{on   }\partial\Omega.
\end{equation}
Moreover there exists a constant $\tilde{C}>0$ such that, for all $0<\lambda <1$,
\begin{equation}\label{bornergo}
 |\lambda u_{\lambda}|\leq \tilde{C} \quad\hbox{in   }\overline{\Omega}.
\end{equation}

\end{lem}

\begin{proof}
 For $R>0$, we consider the following generalized Dirichlet problem
\begin{equation}
\left\{
 \begin{array}{ll}
  -\text{div}\, (|D u_{R,\lambda}|^{p-2} D u_{R, \lambda} )+|D u_{R, \lambda}|^q+\lambda u_{R, \lambda}=f(x) &\quad\text{in}\quad \Omega,\\
u_{R,\lambda}=R& \qquad\text{in}\quad \partial\Omega.
 \end{array}
\right.
\end{equation}
By Theorem \ref{fab}, this problem admits a unique viscosity solution $u_{R,\lambda}$.

 Moreover, $u_{R, \lambda}$ satisfies 
\begin{equation}\label{pbapro}
-\lambda^{-1}\left\| f\right\|_{L^{\infty}} \leq u_{R, \lambda}\leq -\dfrac{M_1}{\beta} d^{\beta} (x)+\dfrac{M_2}{\lambda}\quad\text{in}\quad\Omega.
\end{equation}
Indeed, on the one hand, it is easy to see that $-\lambda^{-1}\left\| f\right\|_{L^{\infty}}$ is a subsolution. On the other hand, borrowing arguments from  \cite{tchamba},  we claim that for some $M_1, M_2>0$ chosen large enough, $\bar{u}(x)=-\dfrac{M_1}{\beta} d^{\beta} (x)+\dfrac{M_2}{\lambda}$ is a supersolution of \eqref{clem1}-\eqref{clem2}. 
Indeed, using that $q(\beta-1)=(p-2)(\beta-1)+(\beta-2)$, we have
\begin{align*}
-\text{div}\, (|D \bar{u}|^{p-2} D \bar{u})+|D \bar{u}|^q+\lambda \bar{u}-\tilde{f}(x)&=M_1^{p-1} |D d|^{p-2} d^{(p-2)(\beta-1)}\Big[ (p-1)(\beta-1)  d^{\beta-2} |D d|^2\\
&+ d^{\beta-1} \Delta d +(p-2)   d^{\beta-1}\left\langle  D^2 d\, \hat{D d}, \hat{Dd}\right\rangle \Big]\\
& +M_1^q d^{q(\beta-1)}| D d|^q-\lambda \dfrac{M_1}{\beta}  d^{\beta} +M_2-\tilde{f}\\
&=M_1^{p-1} |D d|^{p-2} d^{q(\beta-1)}\Big[(p-1) (\beta-1) |D d|^2+d\Delta d\\
&+ (p-2)d \left\langle  D^2 d \,\hat{D d}, \hat{Dd}\right\rangle  +M_1^{q-p+1}| D d|^{q-p+2}\Big]\\
&  -\lambda \dfrac{M_1}{\beta}  d^{\beta} +M_2-\tilde{f}.
\end{align*}
In $\Omega^{\delta}$ where $|D d|=1$ and $0\leq d\leq \delta$, we have
\begin{align*}
 -\text{div}\, (|D \bar{u}|^{p-2} D \bar{u})+|D \bar{u}|^q+\lambda \bar{u}-\tilde{f}(x)&=M_1^{p-1} d^{q(\beta-1)}\Big[(p-1) (\beta-1) +d\Delta d\\
&+ (p-2)d \left\langle  D^2 \ d D\, d, Dd\right\rangle  +M_1^{q-p+1} \\
&-\lambda \dfrac{M_1^{2-p}}{\beta}  d^{\beta(2-p)+p}\Big]
 +M_2-\tilde{f}.
\end{align*}
Taking $M_1>1$ and $M_2>0$ such that 
\begin{eqnarray}\label{condi}
&&M_1^{q-p+1} \geq (p-1)(1-\beta)+ (p-2+\sqrt{N}) \delta\left\| D^2 d\right\|_{L^{\infty}}+\dfrac{\delta^{\beta(2-p)+p}}{\beta} \\
&&\text{and}\quad\quad M_2\geq 2 \left\|\tilde{f}\right\|_{L^{\infty}},
\end{eqnarray}
then  we have $-\text{div}\, (|D \bar{u}|^{p-2} D \bar{u})+|D \bar{u}|^q+\lambda \bar{u}-\tilde{f}(x)\geq 0$ in $ \Omega^{\delta}$.

Now in $\Omega_{\delta}$, we have $|D d |\leq 1$ and $\delta\leq d(x)\leq C(\Omega)$. Using that $0<\lambda<1$, then we have
\begin{eqnarray*}
-\text{div}\, (|D \bar{u}|^{p-2} D \bar{u})+|D \bar{u}|^q+\lambda \bar{u}-\tilde{f}(x) &\geq& M_1^{p-1}\Big[ (p-1)(\beta-1)\left\|d^{(\beta-1)(p-1)-1}\right\|_{L^{\infty}} \\
&&-(p-2+\sqrt{N}) \left\|d^{(\beta-1)(p-1)} \right\|_{L^{\infty}} \left\| D^2 d\right\|_{L^{\infty}}\Big]\\
 &&-\dfrac{M_1}{\beta} \left\| d^{\beta}\right\|_{L^{\infty}} +M_2-\left\|\tilde{f}\right\|_{L^{\infty}}.
\end{eqnarray*}
Hence if we take $M_1$ as in \eqref{condi} and $M_2$ such that
\begin{eqnarray}\label{hihihi}
 M_2&\geq& M_1^{p-1}\Big[ (p-1)(1-\beta)\left\|d^{(\beta-1)(p-1)-1}\right\|_{L^{\infty}}+(p-2+\sqrt{N}) \left\|d^{(\beta-1)(p-1)} \right\|_{L^{\infty}} \left\| D^2 d\right\|_{L^{\infty}}\Big]\nonumber\\
&&+\dfrac{M_1}{\beta} \left\| d^{\beta}\right\|_{L^{\infty}} +3 \left\|\tilde{f}\right\|_{L^{\infty}},
\end{eqnarray}
then the function $\bar{u}$ satisfies the supersolution inequality in $\Omega_{\delta}$.
The estimate  follows by applying the SCR to $-\lambda^{-1}\left\| f\right\|_{L^{\infty}}$, $u_{R,\lambda}$ and $\bar{u}$.

It is worth pointing out that, if $M_2$ is as in \eqref{hihihi}, then 
$$u_{R, \lambda}< R \quad\text{on  }\overline{\Omega}\quad\text{for any   }R>\dfrac{M_2}{\lambda}.$$
 It follows that $u_{R,\lambda}$ is a viscosity solution of \eqref{clem1}-\eqref{clem2} for all $R>\dfrac{M_2}{\lambda}$. Theorem \ref{fab} implies that $u_{\lambda}=u_{R,\lambda}$ for $R>\dfrac{M_2}{\lambda}$.

We have
$$- \max\left(\left\|\tilde{f}\right\|_{L^{\infty}}, M_2\right)\leq \lambda u_{\lambda}\leq \max\left(\left\|\tilde{f}\right\|_{L^{\infty}}, M_2\right).$$
\end{proof}

Now we are in position to prove Theorem~\ref{ergodic}. Using that $u_{\lambda}\geq -\lambda^{-1}\left\| f\right\|_{L^{\infty}}$ in $\overline{\Omega}$,  we have
$$-\text{div} (|D u_{\lambda}|^{p-2} D u_{\lambda})+|D u_{\lambda}|^q-\tilde{f}\leq \left\|\tilde{f}\right\|_{L^{\infty}}\quad\text{in   }\Omega.$$
Theorem \ref{holder} implies uniform   H\"older estimates with respect to  $\lambda$ for the functions $u_{\lambda}$. Consequently if $x_0$ is an arbitrary point in $\overline{\Omega}$, we get that $w_{\lambda}:=u_{\lambda}(x)-u_{\lambda}(x_0)$ is also uniformly bounded in $C^{0, \beta}(\overline{\Omega})$ (recall that $\Omega$ is connected).

From \eqref{bornergo}, we also know that $\left\{-\lambda u_{\lambda}(x_0)\right\}_{\lambda}$ is bounded. 
It follows that, by Ascoli's Theorem, we can extract a uniformly converging subsequence  from $\left\{w_{\lambda}\right\}_{\lambda}$ and we can assume that $\left\{-\lambda u_{\lambda}(x_0)\right\}_{\lambda}$ converges along the same subsequence.  Denoting by $u_{\infty}$ and $c$, the limits of $\left\{w_{\lambda}\right\}_{\lambda}$ and $\left\{-\lambda u_{\lambda}(x_0)\right\}_{\lambda}$ respectively and taking into account that $w_{\lambda}$ solves
$$-\text{div} (|D w_{\lambda}|^{p-2} D w_{\lambda})+|D w_{\lambda}|^q-\tilde{f}(x)+\lambda w_{\lambda}=-\lambda u_{\lambda}(x_0)\quad
\text{in   }\Omega,$$
we can pass into the limit $\lambda\to 0$ and conclude by the stability result for viscosity solutions that, $(c, u_{\infty})$ solves the ergodic problem.

Now let $(c_1, u^1_{\infty})$ and $(c_2, u^2_{\infty})$ be two solutions of the ergodic problem. If $c_1<c_2$ or $c_1>c_2$, we could use Proposition~\ref{ellipscr} to obtain either $u^1_{\infty} \leq u^2_{\infty}$ or $u^2_{\infty} \leq u^1_{\infty}$. But such comparison cannot hold since, for all $k\in \mathbb{R}$, $u^i_\infty +k$ are solutions as well of the ergodic problem, proving the uniqueness of $c$.

\section{ Proof of Theorem \ref{globalex} and Study of the Large Time Behavior}
\subsection{Proof of Theorem \ref{globalex}}
Once one noticed that $u_1(x,t)=t\left\|f\right\|_{L^{\infty}}+\left\|g\right\|_{L^{\infty}}+ \left\|u_0\right\|_{L^{\infty}}$ and $u_2(x,t)=-t\left\|f\right\|_{L^{\infty}}-\left\|g\right\|_{L^{\infty}}- \left\|u_0\right\|_{L^{\infty}}$ are respectively  super-solution  and subsolution
 of \eqref{pb}--\eqref{boundary}, the existence  and uniqueness of a continuous global solution  can be obtained by  Perron's method, combining classical arguments of \cite{userguide} (see also \cite{ishii}), the version up to the boundary of Da Lio \cite{dalio02} and the Strong Comparison Result  of the  Proposition \ref{parascr} on any time interval $[0, T]$. 

\subsection{Large Time Behavior}
Let $u_{\infty}$ be  a bounded solution of \eqref{erg}--\eqref{constr}.
If $c\leq 0$, then $u$ is uniformly bounded. Indeed, if $C> \left\|u_{\infty}\right\|_{L^{\infty}}+\left\|u_0\right\|_{L^{\infty}}+\left\|g\right\|_{L^{\infty}}$, then $u_{\infty}-C$  is a subsolution
 of \eqref{pb}--\eqref{boundary}.   On the other hand, if $\bar{x}$ is a point far enough from $\Omega$, then $|x-\bar{x}|^2 $ is  a super-solution. To see this,it suffices to  take  $\bar{x}$ such that $B(\bar{x}, R)\cap\overline{\Omega}=\emptyset$ with $R> \max (1,\left(\left\|f\right\|_{L^{\infty}} +(p-1)\right)^{\frac{1}{q-p+2}})$.

Hence applying the Strong Comparison Result, we have
$$u_{\infty}(x)-C\leq u(x,t)\leq |x-\bar{x}|^2+ C \quad
\text{on}\quad\overline{\Omega}\times (0, +\infty),$$
and therefore
$$\lim_{t\to\infty} \dfrac{u(x,t)}{t}=0.$$
If $c>0$, then $u_{\infty}-ct +C$ is a supersolution of \eqref{pb}--\eqref{boundary} with state constraint condition on $\partial \Omega$.  On the other hand, $u_{\infty}-ct-C$ is a subsolution
 of \eqref{pb}--\eqref{boundary} which is below $u_0$ at $t=0$ and below $g$ on $\partial\Omega$. Applying the Strong Comparison Result, we have
$$-ct+u_{\infty}-C\leq u(x,t)\leq u_{\infty}-c t +C\quad
\text{on}\quad \overline{\Omega}\times (0, +\infty).$$
The result follows by dividing by $t$ and then letting $t\to + \infty$.

\section*{Appendice: A  General Strong Comparison Result}
\subsection*{A: Properties of the Regularization by Sup-convolution
of Viscosity Subsolutions}

To circumvent the lack of smoothness of the viscosity subsolution
 $u$, we consider instead the more regular  time sup-convolution $u^{\alpha}$. Such regularization was first introduced by Lasry and Lions \cite{lasry1986}, and for $0<\alpha\leq 1$ and $u$ a bounded usc, viscosity subsolution
 is defined by
\begin{equation}\label{defsup}
u^{\alpha}(x,t)=\underset{s\geq 0}{\text{sup}}\left\{u(x,s)-\dfrac{|t-s|^2}{\alpha^2}\right\}.
\end{equation}

 We have the following useful properties on $u^{\alpha}$.
\begin{propo}If $u$ is a bounded usc viscosity subsolution
 $u$  of \eqref{pb}--\eqref{boundary}, the following properties are true
 \begin{enumerate}[i)]
 \item Set $K=\sqrt{2 \left\|u\right\|_{L^{\infty}}}$. Then up to $o_{\alpha}(1)$,
  $u^{\alpha}$ is an usc viscosity subsolution
 of \eqref{pb}--\eqref{boundary} on  $\Omega\times (K\alpha, T- K\alpha)$. Moreover  $u^{\alpha}$ is locally Lipschitz  w.r.t to the time variable and
$$||u_t^{\alpha}||_{L^{\infty}}\leq  \dfrac{2 K}{\alpha} \; .$$
\item  We have $u^{\alpha}(x, K\alpha)\leq u_0(x)+o_{\alpha}(1)$ in $\overline{\Omega}$ and $u^{\alpha}(x,t)\leq g(x,t)+o_{\alpha}(1)$ on $\pO\times (K\alpha, T-K\alpha)$.
\item $u^{\alpha}$ is H\"older continuous w.r.t  the space variable $x$ on $\overline{\Omega}$ uniformly w.r.t the time  for $t> K\alpha$.
 \end{enumerate}
\end{propo}
\begin{proof}
Since $u(x)$ is bounded, the supremum  in \eqref{defsup} is attained  at some point $s^*(t)$ which belong to the interval $(t-K\alpha, t+K \alpha)$. 
Let $\varphi\in C^2\left(\overline{\Omega}\times [K\alpha, T-K\alpha]\right)$ and assume that  $u^{\alpha}-\varphi$ has  a local maximum at $(x_0, t_0)\in \Omega\times(K\alpha, T-K\alpha)$. Denote by $s^*(t_0)$ a point such that $u^{\alpha}(x_0, t_0)=u(x_0, s^*(t_0))- \dfrac{|t_0-s^*(t_0)|^2}{\alpha^2}$, then the function
$$\tau\mapsto u(x, \tau)-\varphi(x, \tau-s^*(t_0)+t_0)$$
reaches a local maximum at $(x_0, s^*(t_0))$. 
Recalling that $u$ is a viscosity subsolution
 of \eqref{pb}, we get by definition
$$\varphi_t(x_0,t_0)-\text{div}\left(|D \varphi|^{p-2} D \varphi(x_0,t_0)\right) +| D \varphi (x_0, t_0)|^q\leq f(x_0, s^*(t_0))\leq f(x_0, t_0)+ o_{\alpha}(1),$$
by using the uniform continuity of $f$ on $\overline \Omega \times [0,T]$. 

Next, let $h>0$ small enough, then
\begin{eqnarray*}
u^{\alpha}(x, t\pm h)-u^{\alpha} (x,t)&\geq& u(x, s^*(t))- \dfrac{|t\pm h-s^*(t)|^2}{\alpha^2} - u(x, s^*(t)) + \dfrac{|t-s^*(t)|^2}{\alpha^2}\\
&=& - \left(\dfrac{h^2\pm 2 h(t-s^*(t))}{\alpha^2}\right)\geq - \left(\dfrac{h^2+2  K h\alpha}{\alpha^2}\right).
\end{eqnarray*}
A first estimate of $u_t^{\alpha}$ (from below) follows by dividing the previous inequality by $h$ and sending $h\to 0$. Exchanging the role of $t+h$ and $t$ provides the estimate from above.

The second assertion  comes from the upper semi-continuity of $u$ and the fact that $u(x,0)\leq u_0(x)$. Indeed $$u^{\alpha}(x, K\alpha)=u(x, s^{*}(K\alpha))-\dfrac{|K\alpha-s^{*}(K\alpha)|^2}{\alpha^2}\leq u(x, s^{*}(K\alpha))$$
 with $s^{*}(K\alpha)\to 0$ as $\alpha\to 0$. Taking the $\lim\,\sup$  we get 
$$\underset{\alpha\to 0}{\text{lim sup}}\, u^{\alpha}(x, K\alpha)\leq \underset{\alpha \to 0}{\text{lim sup}}\, u(x, s^{*}(K\alpha))\leq u(x,0)\leq u_0(x).$$
Similarly, we use the semi-continuity of $u$ and Proposition \ref{boundcond} to prove that
$$u^{\alpha}(x,t)\leq g(x,t)+o_{\alpha}(1).$$

The last assertion is a consequence of Theorem \ref{holder}  where  $C=\left\|f\right\|_{L^{\infty}}+\left\|u^{\alpha}_t\right\|_{L^{\infty}}$.
\end{proof}

 Let us note that, using the lower semi-continuity of $v$ and $v(x,0)\geq u_0(x)$, we have $v(x, K\alpha)\geq u_0(x)-o_{\alpha}(1)$. Hence \begin{equation}\label{prem}
       u^{\alpha}(x,K\alpha)\leq v(x, K\alpha)+ \omega(\alpha) \quad\text{for all}\quad x \in\overline{\Omega},
      \end{equation}
for some $\omega(\alpha)$ satisfying $\underset{\alpha\to 0}{\lim}\, \omega(\alpha)=0$.

\end{proof}
\subsection*{B: Proof of Proposition \ref{parascr}}

In order to prove the SCR, we are going to show that $\tilde{u}^{\alpha}-v\leq \omega(\alpha)$ in $\overline{\Omega}\times [K \alpha, T-K\alpha]$. Inequality \eqref{compar}  follows by passing to the limit as  $\alpha\to 0$. To do so, the continuity of $u^{\alpha}$ is a key point since it allows to use the  arguments of \cite{rouy, barle1}.

For the sake of simplicity of notations, we drop  the $\widetilde{}$ on $\tilde{u}^{\alpha}$. The key idea is to compare  $u^{\alpha}_{\mu}:=\mu u^{\alpha}$ and $v$ with $0<\mu<1$ close to 1 in order to take care of the difficulty due to the $|D u|^q$ term. 

We argue by contradiction assuming that $M^{\alpha}=\underset{\overline{\Omega}\times[K\alpha, T-K\alpha]}{\mathrm{max}} (u^{\alpha}-v-\omega(\alpha))>0$. If $\mu$ is sufficiently close to 1 and if $\eta_{\alpha}>0$ is a constant small enough, then we have  $M_{\mu, \eta}^{\alpha}=\underset{\overline{\Omega}\times[K\alpha, T-K\alpha]}{\mathrm{max}} \left( u^{\alpha}_{\mu}-v -\omega(\alpha)-\eta_{\alpha} (t-K\alpha)\right)>M^{\alpha}/2$.

 We denote by $(x_0, t_0)$ a point of $\overline{\Omega}\times [K\alpha, T-K\alpha]$ such that $M_{\mu, \eta}^{\alpha}= u_{\mu}^{\alpha}(x_0,t_0)-v(x_0,t_0) -\omega(\alpha)-\eta_{\alpha} (t_0-K\alpha)$.  The existence of $(x_0,t_0)$ is guaranteed by the upper and lower semi-continuity of $u^{\alpha}$ and $v$ respectively (we drop the dependence of $(x_0, t_0)$ on  $\eta_{\alpha}$, $\alpha$ and $\mu$ for the sake of simplicity of notations).  Since $M_{\eta, \mu}^{\alpha}>0$, we necessarily have $t_0 > K\alpha$ in view of  \eqref{prem}. By the Maximum Principle of the ''Users guide'' \cite{userguide}, we have $(x_0, t_0) \in\partial\Omega\times (K\alpha, T-K\alpha)$.\\

Next, using the regularity of the boundary, we can find a $C^2$-function $\xi:\mathbb{R}^N\rightarrow\mathbb{R}^N$  which is equal to $n=-Dd$ in a neighborhood of $\partial \Omega$.
Now we consider the auxiliary function $\Phi_{\eps}:\overline{\Omega}\times\overline{\Omega}\times [K\alpha, T-K\alpha]\times[K\alpha, T-K\alpha]\rightarrow\mathbb{R}$  defined by
\begin{eqnarray*}
\Phi_{\eps}(z,w,t,s)&=u_{\mu }^{\alpha}(z,t)-v(w,s)-\omega(\alpha)-\eta_{\alpha}(t-K\alpha)-\left|\dfrac{z-w}{\varepsilon}-\chi\left(\dfrac{z+w}{2}\right)\right|^4\\
&+\dfrac{|t-s|^2}{\varepsilon^2}
\end{eqnarray*}

Let $(\bar{z}, \overline{w}, \bar{t}, \bar{s})$ be  a global maximum point of $\Phi_{\eps}$ on $\overline{\Omega}\times\overline{\Omega}\times [K\alpha, T-K\alpha]\times[K\alpha, T-K\alpha]$. For notational simplicity  we drop again the dependance  of $(\bar{z}, \overline{w}, \bar{t}, \bar{s})$  on $\eps, \mu$ and $\eta$. Using the inequality $\Phi_{\eps}(\bar{z}, \overline{w}, \bar{t}, \bar{s})\geq \Phi_{\eps}(x_0, x_0, t_0, t_0)$ and the  boundedness of $u_{\mu }^{\alpha}, v$ and  $\chi$, we have
$$\left|\dfrac{\bar{z}-\overline{w}}{\varepsilon}\right|\leq C, \quad
\left| \dfrac{\bar{t}-\bar{s}}{\varepsilon}\right|\leq C,$$ 
for some  constant $C>0$ depending on $\left\|u\right\|_{L^{\infty}}, \left\|v\right\|_{L^{\infty}}$ and $\alpha$. By the compactness of $\overline{\Omega}\times [K\alpha, T-K\alpha]$, we can assume that $(\bar{z},\bar{t})$, $(\overline{w},  \bar{s})$ converge to $(\tilde{x}, \tilde{t})\in \overline{\Omega}\times [K\alpha, T-K\alpha]$. Moreover, using the continuity of $u^{\alpha}$, we have
\begin{eqnarray*}
 \Phi_{\eps}(\bar{z}, \overline{w}, \bar{t}, \bar{s})
\geq \Phi_{\eps}(x_0-\eps\xi(x_0), x_0, t_0, t_0)=M_{\mu, \eta}^{\alpha}- o_{\eps}(1),\quad\text{as}\, \eps\to 0,
\end{eqnarray*}
and hence  
\begin{equation}\label{pet}
\underset{\eps\to 0}{\text{lim inf}}\, \Phi_{\eps}(\bar{z}, \overline{w}, \bar{t}, \bar{s})\geq M_{\mu, \eta}^{\alpha}.
\end{equation}

On the other hand, we have also
\begin{eqnarray}
\underset{\eps\to 0}{\text{lim sup}}\, \Phi_{\eps}(\bar{z}, \overline{w}, \bar{t}, \bar{s})& \leq& \underset{\eps\to 0}{\text{lim sup}}\,( u_{\mu }^{\alpha}(\bar{z}, \bar{t})-v(\overline{w}, \bar{s})-\eta_{\alpha} (\bar{t}-K\alpha)-\omega(\alpha))\nonumber\\
& \quad &-\underset{\eps\to 0}{\text{lim inf}}\,\left|\dfrac{\bar{z}-\overline{w}}{\varepsilon}-\chi\left(\dfrac{\bar{z}+\overline{w}}{2}\right)\right|^4\nonumber\\
& \quad &-\underset{\eps\to 0}{\text{lim inf}}\,\dfrac{|\bar{t}-\bar{s}|^2}{\varepsilon^2}\nonumber\\
&\leq & M_{\mu,\eta}^{\alpha}\label{set}.
\end{eqnarray}

Therefore,  combining \eqref{pet} and \eqref{set} with classic arguments, we have
\begin{equation}
\left|\dfrac{\bar{z}-\overline{w}}{\varepsilon}-\chi\left(\dfrac{\bar{z}+\overline{w}}{2}\right)\right|^4=o_{\eps}(1),\qquad \dfrac{|\bar{t}-\bar{s}|^2}{\varepsilon^2}=o_{\eps}(1),\label{vet}
\end{equation}
\begin{eqnarray}
 u_{\mu }^{\alpha}(\bar{z} \bar{t})-v(\overline{w}, \bar{s})-\eta_{\alpha} (\bar{t}-K\alpha)-\omega(\alpha)&\to& 
 u_{\mu }^{\alpha}(\tilde{x}, \tilde{t})-v(\tilde{x}, \tilde{t})-\eta_{\alpha} (\tilde{t}-K\alpha)-\omega(\alpha)\nonumber\\
&=&M_{\mu, \eta}^{\alpha} \quad \text{as}\, \eps\to  0.
\end{eqnarray}

It follows that $u_{\mu }^{\alpha}(\bar{z}, \bar{t})\to u_{\mu }^{\alpha}(\tilde{x}, \tilde{t})$ and $v(\overline{w}, \bar{s})\to v(\tilde{x}, \tilde{t})$.

Now, recalling the properties of $u^{\alpha}$ and $v$ at $t=K\alpha$, we have $\bar{t}, \bar{s}>K\alpha$ for $\eps$ small enough. Next we claim that, for $\eps$ small enough the viscosity inequalities hold for $u^{\alpha}$ and $v$.  This is obviously the case for $v$ if $\overline{w}\in\Omega$. If on the contrary $\overline{w}\in \partial\Omega$, then we necessarily  have $v(\overline{w}, \bar{s})<g(\overline{w}, \bar{s})$.
Indeed if $\overline{w}\in \partial\Omega$ then $\tilde{x}\in \pO$. Since there is no loss of boundary conditions for subsolution
s as clearly specified in Proposition \ref{boundcond}, we have
$$\mu u^{\alpha}(\tilde{x}, \tilde{t}) \leq \mu (g(\tilde{x}, \tilde{t})+\omega(\alpha)).$$
Using that $M_{\mu, \eta}^{\alpha}>0$, we cannot have $v(\tilde{x}, \tilde{t})\geq g(\tilde{x},\tilde{t})$
since we 
would then have
$$\dfrac{M^{\alpha}}{2}\leq M_{\mu, \eta}^{\alpha}=\mu u^{\alpha}(\tilde{x}, \tilde{t})-v(\tilde{x}, \tilde{t})-\eta_{\alpha} (\tilde{t}-K\alpha)-\omega(\alpha) \leq( \mu-1)( g(\tilde{x}, \tilde{t})+\omega(\alpha)),$$
 a contradiction by sending $\mu \to 1$.\\

It follows that, if $\tilde{x}\in\partial\Omega$, then we have necessarily that
\begin{equation}
v(\tilde{x}, \tilde{t})< g(\tilde{x},\tilde{t})\quad\text{and}\quad\mu u^{\alpha}(\tilde{x}, \tilde{t}) \leq \mu (g(\tilde{x}, \tilde{t})+\omega({\alpha})).
\end{equation}
 Hence, using that $v(\overline{w}, \bar{s})\to v(\tilde{x}, \tilde{t})<g(\tilde{x}, \tilde{t})$, we deduce that if $\overline{w}\in\partial\Omega$, then $v(\overline{w}, \bar{s})< g(\overline{w}, \bar{s})$ for $\eps$ small enough and the viscosity inequality holds also in this case.

On the other hand, from \eqref{vet} we get that
\begin{equation}
\bar{z}=\overline{w}+\varepsilon\chi\left(\dfrac{\bar{z}+\overline{w}}{2}\right)+ o_{\eps}(1),
\end{equation}
 which implies by the smoothness of the domain and the properties of $\chi$ that $\bar{z}$ lies in  $\Omega$ for $\eps$ small enough and hence the viscosity inequality for $u_{\mu}^{\alpha}$ holds too.\\

Next, we notice
that $ u^{\alpha}_{\mu}$  satisfies
$$\dfrac{1}{\mu}  (u^{\alpha}_{\mu})_t-\dfrac{1}{\mu^{p-1}} \text{div}\,(|D u^{\alpha}_{\mu}|^{p-2} D u^{\alpha}_{\mu})+\dfrac{1}{\mu^q} |D u_{\mu}^{\alpha}|^q\leq f+o_{\alpha}(1) \quad\text{in   } \Omega\times (K\alpha, T-K\alpha),$$
and we can also re-write it as 
\begin{eqnarray*}
&\mu^{p-2}(u^{\alpha}_{\mu})_t-|D u^{\alpha}_{\mu}|^{p-2} \left\{\Delta u^{\alpha}_{\mu}+(p-2) (D^2 u^{\alpha}_{\mu} \widehat{D u^{\alpha}_{\mu}}, \widehat{ Du^{\alpha}_{\mu}})-\mu^{p-1-q} |D u_{\mu}^{\alpha}|^{q-p+2}\right\}\\
&\qquad\qquad\leq \mu^{p-1} (f+o_{\alpha}(1)),
\end{eqnarray*}
where $\widehat{\xi}=\dfrac{\xi}{|\xi|}$ for $\xi\neq 0$ and $\widehat{\xi}=0$ if $\xi\equiv 0$.\\

 The Jensen-Ishii's Lemma \cite{userguide} ensures the existence of $X$, $Y\in \mathcal{S}^N$, $a$, $b\in \mathbb{R}$, $q_1$, $q_2\in\mathbb{R}^N$ such that 
\begin{equation}
(a,q_1,X)\in \overline{\mathcal{P}}^{2,1,+} u_{\mu}^{\alpha}(\bar{z}, \bar{t}), \qquad\qquad 
(b,q_2,Y)\in \overline{\mathcal{P}}^{2,1,-} v(\overline{w}, \bar{s}),
\end{equation}
\begin{eqnarray}\label{grad}
&a-b\geq\eta_{\alpha}>0, \qquad\qquad |q_1-q_2|\leq C\varepsilon (|q_1|\wedge |q_2|),\\
&- \dfrac{o(1)}{\varepsilon^2} I_{2N}\leq 
\begin{pmatrix}
   X & 0 \\
   0 & -Y 
\end{pmatrix}
\leq \dfrac{o(1)}{\varepsilon^2} \begin{pmatrix}
   I_N&-I_N \\
   -I_N & I_N 
\end{pmatrix}+o(1) I_{2N}.\label{matrice}
\end{eqnarray}
where $|q_1|\wedge |q_2|$ denotes the minimum of $|q_1|$ and $|q_2|$. Indeed, for (\ref{grad}), we remark that 
$$q_1 = \left (I_N - \frac{\eps }{2} D\chi\left(\dfrac{\bar{z}+\overline{w}}{2}\right)\right)q \quad\hbox{and}\quad q_2 = \left (I_N + \frac{\eps }{2} D\chi\left(\dfrac{\bar{z}+\overline{w}}{2}\right)\right)q\,,$$
with 
$$q = \frac{4}{\eps} \left|\dfrac{z-w}{\varepsilon}-\chi\left(\dfrac{z+w}{2}\right)\right|^2\left(\dfrac{z-w}{\varepsilon}-\chi\left(\dfrac{z+w}{2}\right)\right)\; ,$$
and (\ref{grad}) is an easy consequence of the boundedness of $D\chi$.

Moreover the viscosity inequalities for $u_{\mu}^{\alpha}$ and $v$ read
\begin{eqnarray}\label{visc1}
&\mu^{p-2}a-|q_1|^{p-2} \left\{ \text{tr}\big(\left[Id+(p-2) \left(\hat{q_1}\otimes \hat{q_1}\right)\right] X\big)-\mu^{p-1-q} |q_1|^{q-p+2}\right\}\\
&\qquad\qquad\leq \mu^{p-1}
(f(\bar{z},\bar{t})+o_{\alpha}(1))\nonumber,\\
&b-|q_2|^{p-2} \left\{\text{tr}\big(\left[Id+(p-2) (\hat{q_2}\otimes  \hat{ q_2})\right]Y\big)-|q_2|^{q-p+2}\right\}\geq f(\bar{w},\bar{s}).\label{visc2}
\end{eqnarray}

In the sequel we fix  $\eta_{\alpha}> 2 o_{\alpha}(1)$ (recall that the $o_{\alpha}(1)$ comes from the sup-convolution procedure and is fixed, therefore we can choose in such a way $\eta_{\alpha}$).
Since  we may have a singularity at  $q_1=0$ or $q_2=0$, we have to consider separately three cases. First we assume that there exists a constant $\gamma>0$ such  that $$|q_1|, |q_2|\geq \gamma.$$ In this case the matrix $A(\xi)=Id+(p-2) (\hat{\xi}\otimes  \hat{ \xi})$ is positive definite, so that its matrix square root $\sigma$ exists and satisfies
$$|\sigma(\xi_1)-\sigma (\xi_2)|\leq c\dfrac{|\xi_1-\xi_2|}{|\xi_1|\wedge|\xi_2|}.$$
Combining \eqref{grad} with the fact  that \eqref{matrice} implies  that $X\leq Y+o_{\eps}(1)$, we have
\begin{equation}
\text{tr}\left(A(q_1)X\right)-\text{tr}\left(A(q_2)Y\right)
\leq \dfrac{o_{\eps}(1)}{\varepsilon^2} |\sigma(q_1)-\sigma(q_2)|^2+o_{\eps}(1)\leq o_{\eps}(1),
\end{equation}
\begin{eqnarray}
|q_2|^{q-p+2}-\mu^{p-1-q} |q_1|^{q-p+2}&=&|q_2|^{q-p+2}-|q_1|^{q-p+2}+(1-\mu^{p-1-q}) |q_1|^{q-p+2}\nonumber\\
&\leq& (q-p+2)|q_1|^{q-p+1}|q_2-q_1|+(1-\mu^{p-1-q}) |q_1|^{q-p+2}\nonumber\\
&\leq&  o(\varepsilon)|q_1|^{q-p+2}+(1-\mu^{p-1-q}) |q_1|^{q-p+2}.
\end{eqnarray}
Multiplying \eqref{visc1} by $\dfrac{|q_2|^{p-2}}{|q_1|^{p-2}}$ which is of order $1+O(\varepsilon)$ and subtracting  from it  \eqref{visc2}, we have

\begin{eqnarray}
(1+O(\varepsilon)) \mu^{p-2}a- b &\leq&|q_2|^{p-2} \left\{ o_{\eps}(1)+o(\varepsilon)|q_1|^{q-p+2}+(1-\mu^{p-1-q}) |q_1|^{q-p+2}\right\}\nonumber\\
&+& \mu^{p-1}
(1+O(\varepsilon))(f(\bar{z},\bar{t})+o_{\alpha}(1))- f(\overline{w},\bar{s}).
\end{eqnarray}
At this point, we recall that the Lipschitz continuity of $u^\alpha$ implies that $|a|\leq \dfrac{2 \mu K}{\alpha}$. On the other hand we remark that, since $1-\mu^{p-1-q} <0$, for fixed $\mu$ the term $o(\varepsilon)|q_1|^{q-p+2}$ is controlled by the $(1-\mu^{p-1-q}) |q_1|^{q-p+2}$ term. 

Now we are going to let $\eps\to 0$ : if we assume that $q_1,q_2$ (which depend on $\eps$) are bounded, we may assume that they converge (we still denote their limits as $q_1,q_2$ respectively). For $\mu$ close enough to 1, we get as $\eps\to 0$
$$0<\eta_{\alpha}/2\leq \mu^{p-2} a-b\leq (\mu^{p-1} -1) f(\tilde{x}, \tilde{t})+ \mu^{p-1}o_{\alpha}(1)+|q_2|^{p-2}(1-\mu^{p-1-q}) |q_1|^{q-p+2}$$
Recalling that  $\eta_{\alpha}>2 o_{\alpha}(1)$, we get a contradiction when $\mu\to 1$ since the last term of the right-hand side is negative. Of course, we get the same contradiction if (at least for some subsequence) $q_1$ or $q_2 \to \infty$.\\

If $q_1, q_2\neq 0$ but $q_1\to0$, $q _2\to 0$  then, noticing that $\dfrac{|q_2|^{p-2}}{|q_1|^{p-2}}$ is still of order $1+O(\varepsilon)$, we can pass to the limit $\eps\to 0$ in the same way and obtain
$$0<\eta_{\alpha}/2\leq \mu^{p-2} a-b\leq (\mu^{p-1} -1) f(\tilde{x}, \tilde{t})+ \mu^{p-1}o_{\alpha}(1),$$
and we also get a contradiction.\\

If $q_1=0$ or $q_2=0$, then necessarily $q_1=q_2=0$  and, by subtracting \eqref{visc2} from \eqref{visc1}, we have
$$\eta_{\alpha}/2\leq \mu^{p-2}a-b\leq \mu^{p-1}
(f((\bar{z},\bar{t}))+o_{\alpha}(1))- f(\overline{w},\bar{s}).$$
We get a contradiction when $\eps\to 0$.

In all cases fixing $\eta_{\alpha}> 2 o_{\alpha}(1)$ we get  a contradiction for  $\eps$ small enough and $\mu$ close to 1 and the conclusion follows.

\subsection*{ C: Proof of Proposition \ref{ellipscr}}
The proof of $(i)$ and $(ii)$ are very similar. Indeed we know by Theorem \ref{holder} that subsolutions of \eqref{quasistation} are  H\"older continuous, so we  are always in a case where the subsolution
 or the supersolution are continous.
We will only give  details of  the proof  for the $p-Laplacian$ operator  in the case where $\lambda>0$ and $v\in C(\overline{\Omega})$.   The other cases are an easy adaptation (the equation \eqref{operat} is even easier to study).  Since $v$ is assumed to be continuous,  we follow the proof of \cite{rouy} with the same trick as before in order to take care of the strong growth of the gradient term.
We argue by contradiction assuming that $M=\underset{\overline{\Omega}}{\mathrm{max}}\, (u-v)>0$ 
If $\mu$ is sufficiently close to 1, then we have  $M_{\mu}=\underset{\overline{\Omega}}{\mathrm{max}}\, ( u_{\mu}-v)>M/2>0$.
Since $u$ is usc and $v$ is continuous this maximum is achieved at $x_0$. We may assume that $x_0\in\partial\Omega$. We drop the dependence of $x_0$ on $\mu$.

Next, using the regularity of the boundary, we can find a $C^2$-function $\xi:\mathbb{R}^N\rightarrow\mathbb{R}^N$  which is equal to $n$ in a neighborhood of $\partial \Omega$.
Now we consider the test function $\Phi_{\eps}:\overline{\Omega}\times\overline{\Omega}\rightarrow\mathbb{R}$ defined by
\begin{equation*}
 \Phi_{\eps}(z,w)=\mu u(z)-u(w)+\left|\dfrac{z-w}{\eps}+\xi\left(\dfrac{z+w}{2}\right)\right|^4.
\end{equation*}

Let $(x,y)$ be  a global maximum  point of $\Phi_{\eps}$ on $\overline{\Omega}\times\overline{\Omega}$. For notational simplicity we drop the dependence of $x$ and $y$ on $\eps$. and $\mu$.
Using the boundedness of $u$ and $v$, it is clear that $x-y=O(\eps)$ and it follows that, along  a subsequence, $x,y\to \bar{x}\in \overline{\Omega}$.
Since $u$ is lsc and $v$ is continuous, we get that $\underset{\eps\to 0}{ \limsup}\, \Phi_{\eps}(x,y)\leq M_{\mu}$.

On the other hand we have, using the \textbf{continuity} of $v$, we have
$$\Phi_{\eps}(x,y)\geq \Phi_{\eps}(x_0, x_0-\eps\xi(x_0))\geq M_{\mu}+O(\eps).$$
It follows that $\underset{\eps\to 0}{\liminf}\, \Phi_{\eps}\leq M_{\mu}$.  Hence we get that 
\begin{equation}\label{fifi}
 \Phi_{\eps}(x,y)\to M_{\mu}\quad\text{as}\quad\eps\to 0.
\end{equation}

Standard arguments allow us  to deduce from \eqref{fifi} that
\begin{eqnarray}
\left|\dfrac{x-y}{\eps}+\xi\left(\dfrac{x+y}{2}\right)\right|^4=o_{\eps}(1),\\
 \mu u(x)-v(y)\to \mu u(\bar{x})-v(\bar{x})= M_{\mu}\quad \text{as}\quad\eps\to 0.\label{titi}
\end{eqnarray}

Next we claim that, for $\eps$ small enough the viscosity inequalities hold for $u$ and $v$. This is obviously the case if $y\in \Omega$. Using Proposition~\ref{boundcond} and arguing similarly as in the previous proof, we get that, if $y\in\partial\Omega$ than $v(y)\leq \tilde{g}(y)$ and the viscosity inequality holds also in this case. On the other hand,  using \eqref{titi}, 
we get  that 
\begin{equation}
 x=y-\eps\xi(y)+o_{\eps}(1)
\end{equation}
which implies by the smoothness of the domain and the properties of $\xi$ that $x$ lies in $\Omega$ for $\eps$ small enough and hence the viscosity  inequality holds for $\mu u$.

Using the same arguments as the previous  proof, we get that the elements $(q_1, X)\in \overline{\mathcal{J}}^{2,+} u_{\mu}$ and $(q_2, Y)\in \overline{\mathcal{J}}^{2,-} v$ given by the Jensen-Ishii's Lemma satisfy
\begin{eqnarray}
\lambda(1+O(\varepsilon)) \mu^{p-2}u_{\mu}(x)- v(y) &\leq&|q_2|^{p-2} \left\{ o_{\eps}(1)+o(\varepsilon)|q_1|^{q-p+2}+(1-\mu^{p-1-q}) |q_1|^{q-p+2}\right\}\nonumber\\
&+& \mu^{p-1}
(1+O(\varepsilon))\tilde{f}(x)- \tilde{f}(y).
\end{eqnarray}
Letting $\eps\to 0$ and then $\mu\to 1$, we get a contradiction. It follows that $\tilde{u}\leq v$ on $\overline{\Omega}$.

\nocite{*}
\bibliography{bibliodege}
\bibliographystyle{plain}
\end{document}